\documentclass[a4, 12pt]{amsart}
\usepackage[mathscr]{eucal}
\usepackage{amssymb}
\usepackage{latexsym}
\usepackage{amsthm}
\usepackage{color}
\theoremstyle{plain}
\newtheorem{theorem}{Theorem}[section]

\newtheorem{lemma}{Lemma}[section]
\newtheorem{proposition}{Proposition}[section]

\makeatletter
\@addtoreset{equation}{section}

\title[A rigidity theorem of self-expander]
{A rigidity theorem of self-expander}

\author [Z. Li and G. Wei]{Zhi Li and Guoxin Wei}

\address{Zhi Li \\  \newline \indent College of Mathematics and Information Science, Henan Normal University,
\newline \indent 453007, Xinxiang, Henan, China.}
\email{lizhihnsd@126.com}

\address{Guoxin Wei \\ \newline \indent School of Mathematical Sciences, South China Normal University,
\newline \indent 510631, Guangzhou,  China.}
\email{weiguoxin@tsinghua.org.cn}

\begin{document}
\maketitle

\begin{abstract}
In this paper, we completely classify $3$-dimensional complete self-expanders with constant norm $S$ of the second fundamental form and constant $f_{3}$ in Euclidean space $\mathbb R^{4}$, where $h_{ij}$
are components of the second fundamental form, $S=\sum_{i,j}h^{2}_{ij}$ and $f_{3}=\sum_{i,j,k}h_{ij}h_{jk}h_{ki}$.
\end{abstract}


\section{introduction}
\vskip2mm
\noindent

Let $X: M^{n}\to \mathbb{R}^{n+p}$ be an $n$-dimensional  submanifold in the $(n+p)$-dimensional Euclidean space
$\mathbb{R}^{n+p}$.
If  a family of hypersurfaces   for $t \in [0, T)$,  satisfies
$$
\dfrac{\partial X(t)}{\partial t}=\vec{H(t)}, \ \  X(0)=X,
$$
it is called mean curvature flow (MCF), where $\vec H(t)$ denotes mean curvature vector.
An $n$-dimensional smooth immersed submanifold $X: M^{n}\to \mathbb{R}^{n+p}$ is called self-expanders of MCF if its mean curvature vector $\vec H$ satisfies the equation
\begin{equation}\label{1.1-1}
\vec H= X^{\perp},
\end{equation}
where $X^{\perp}$ and $\vec H$ denote the normal component of the position vector $X$ and mean curvature vector, respectively. $M ^{n}$ is a self-expander if and only if $\sqrt{2t} M ^{n}$ is a MCF, $t\in(0,\infty)$.

Self-expanders have a very important role in the study of MCF. They describe the asymptotic longtime behavior for MCF and the local structure of MCF after the singularities in
the very short time. In the case of codimension one and under certain assumptions on the initial hypersurface at infinity, Ecker and Huisken \cite{EH}
showed that the solutions of mean curvature flow of
entire graphs in euclidean space exist for all times $t>0$ and converges to a self-expander. In \cite{Sta}, Stavrou
proved the same result under weaker hypotheses that the initial hypersurface has a unique
tangent cone at infinity. Self-expanders also appear in the mean curvature evolution of cones.  Ilmanen \cite{Ilm} studied
the existence of E-minimizing self-expanding hypersurfaces which converge
to prescribed closed cones at infinity in Euclidean space. Ding \cite{D} studied self-expanders and their
relationship to minimal cones in Euclidean space. Bernstein and Wang (\cite{BW1} and \cite{BW2}) obtained various results on asymptotically conical self-expanders.
In \cite{CZ}, Cheng and Zhou studied some properties of complete properly immersed self-expanders and proved some results related to the spectrum of the drifted Laplacian for self-expanders in higher codimension.
Recently, Ancari and Cheng \cite{AC} mainly studied self-expander hypersurfaces in Euclidean space whose mean curvatures have some linear growth controls. Ishimura \cite{Ish} and Halldorsson \cite{H} have completed the classification problem of self-expander curves in $\mathbb{R}^{2}$. Ancari and Cheng \cite{AC} prove that the surfaces $\Gamma\times \mathbb{R}$ with the product metric are the
only complete self-expander surfaces immersed in $\mathbb{R}^{3}$ with constant scalar curvature, where $\Gamma$ is a complete self-expander curve immersed in $\mathbb{R}^{2}$.
There are many other works about self-expanders (see \cite{AIC}, \cite{FM}, \cite{Smo}, \cite{XY} and references therein).

Recently, by studying the infimum of $H^{2}$, Li and Wei \cite{LW2} gave a classification for complete self-expander surfaces immersed in $\mathbb{R}^{3}$ with constant squared norm of the second fundamental form by making use of the generalized maximum principle.
For the higher dimension $n$, it is not easy to classify self-expander in Euclidean space with
constant squared norm of the second fundamental form.
In this paper, for $n=3$, by studying the infimum of $H$, we prove
\begin{theorem}\label{theorem 1.1}
Let $X: M^{3}\to \mathbb{R}^{4}$ be a
$3$-dimensional complete self-expanders in $\mathbb R^{4}$.
If the squared norm $S$ of the second fundamental form and $f_{3}$ are constant, then $X: M^{3}\to \mathbb{R}^{4}$ is a hyperplane
$\mathbb {R}^{3}$ through the origin, where $h_{ij}$
are components of the second fundamental form, $S=\sum_{i,j}h^{2}_{ij}$ and $f_{3}=\sum_{i,j,k}h_{ij}h_{jk}h_{ki}$.
\end{theorem}

\vskip5mm
\section {Preliminaries}
\vskip2mm

\noindent

Let $X: M^{n} \rightarrow\mathbb{R}^{n+1}$ be an
$n$-dimensional hypersurface of $(n+1)$-dimensional Euclidean space $\mathbb{R}^{n+1}$. We choose a local orthonormal frame field
$\{e_{A}\}_{A=1}^{n+1}$ in $\mathbb{R}^{n+1}$ with dual coframe field
$\{\omega_{A}\}_{A=1}^{n+1}$, such that, restricted to $M^{n}$,
$e_{1},\cdots, e_{n}$  are tangent to $M^{n}$. From now on,  we use the following conventions on the ranges of indices,
$$
 1\leq i,j,k,l\leq n.
$$
 $\sum_{i}$ means taking  summation from $1$ to $n$ for $i$.
Then we have
\begin{equation*}
dX=\sum_i\limits \omega_{i} e_{i}, \ \ de_{i}=\sum_{j}\limits \omega_{ij}e_{j}+\omega_{i n+1}e_{n+1},
\end{equation*}
\begin{equation*}
de_{n+1}=\omega_{n+1 i}e_{i}, \ \ \omega_{n+1i}=-\omega_{in+1},
\end{equation*}
where $\omega_{ij}$ is the Levi-Civita connection of the hypersurface.

\noindent By  restricting  these forms to $M$,  we get
\begin{equation}\label{2.1-1}
\omega_{n+1}=0.
\end{equation}

\noindent Taking exterior derivatives of \eqref{2.1-1}, we obtain
\begin{equation}\label{2.1-2}
\omega_{in+1}=\sum_{j} h_{ij}\omega_{j},\quad
h_{ij}=h_{ji}.
\end{equation}

$$
h=\sum_{i,j}h_{ij}\omega_i\otimes\omega_j, \ \ H= \sum_i\limits h_{ii}
$$
are called  second fundamental form and mean curvature  of $X: M\rightarrow\mathbb{R}^{n+1}$, respectively.
Let $S=\sum_{i,j}\limits (h_{ij})^2$ be  the squared norm
of the second fundamental form  of $X: M\rightarrow\mathbb{R}^{n+1}$.
The induced structure equations of $M$ are given by
\begin{equation*}
d\omega_{i}=\sum_j \omega_{ij}\wedge\omega_j, \quad  \omega_{ij}=-\omega_{ji},
\end{equation*}
\begin{equation*}
d\omega_{ij}=\sum_k \omega_{ik}\wedge\omega_{kj}-\frac12\sum_{k,l}
R_{ijkl} \omega_{k}\wedge\omega_{l},
\end{equation*}
where $R_{ijkl}$ denote components of the curvature tensor of the hypersurface.
Hence,
the Gauss equations are given by
\begin{equation}\label{2.1-3}
R_{ijkl}=h_{ik}h_{jl}-h_{il}h_{jk}.
\end{equation}

\noindent
Defining the
covariant derivative of $h_{ij}$ by
\begin{equation}\label{2.1-4}
\sum_{k}h_{ijk}\omega_k=dh_{ij}+\sum_kh_{ik}\omega_{kj}
+\sum_k h_{kj}\omega_{ki}.
\end{equation}
we obtain the Codazzi equations
\begin{equation}\label{2.1-5}
h_{ijk}=h_{ikj}.
\end{equation}
By taking exterior differentiation of \eqref{2.1-4}, and
defining
\begin{equation}\label{2.1-6}
\sum_lh_{ijkl}\omega_l=dh_{ijk}+\sum_lh_{ljk}\omega_{li}
+\sum_lh_{ilk}\omega_{lj}+\sum_l h_{ijl}\omega_{lk},
\end{equation}
we have the following Ricci identities
\begin{equation}\label{2.1-7}
h_{ijkl}-h_{ijlk}=\sum_m
h_{mj}R_{mikl}+\sum_m h_{im}R_{mjkl}.
\end{equation}
Defining
\begin{equation}\label{2.1-8}
\begin{aligned}
\sum_mh_{ijklm}\omega_m&=dh_{ijkl}+\sum_mh_{mjkl}\omega_{mi}
+\sum_mh_{imkl}\omega_{mj}+\sum_mh_{ijml}\omega_{mk}\\
&\ \ +\sum_mh_{ijkm}\omega_{ml}
\end{aligned}
\end{equation}
and taking exterior differentiation of  \eqref{2.1-6}, we get
\begin{equation}\label{2.1-9}
\begin{aligned}
h_{ijklq}-h_{ijkql}&=\sum_{m} h_{mjk}R_{milq}
+ \sum_{m}h_{imk}R_{mjlq}+ \sum_{m}h_{ijm}R_{mklq}.
\end{aligned}
\end{equation}

\noindent The $\mathcal{L}$-operator is defined by
\begin{equation*}
\mathcal{L}f=\Delta f+\langle X,\nabla f\rangle,
\end{equation*}
where $\Delta$ and $\nabla$ denote the Laplacian and the gradient
operator, respectively.
\noindent
We define the function $f_{3}$ as follows:
$$f_{3}=\sum_{i,j,k}h_{ij}h_{jk}h_{ki},$$
then we get
\begin{equation}\label{2.1-10}
\aligned
&\nabla_{l}f_{3}=3\sum_{i,j,k}h_{ijl}h_{jk}h_{ki}, \\
&\nabla_{p}\nabla_{l}f_{3}=3\sum_{i,j,k}h_{ijlp}h_{jk}h_{ki}+6\sum_{i,j,k}h_{ijl}h_{jkp}h_{ki}, \ \ \ l,p=1, 2, \cdots, n.
\endaligned
\end{equation}

\noindent We next suppose that  $X: M\rightarrow\mathbb{R}^{n+1}$
is a self-expanders, that is, $H=\langle X, e_{n+1}\rangle$. By a simple calculation, we have the following basic formulas.

\begin{equation}\label{2.1-11}
\aligned
\nabla_{i}H
=&-\sum_{k}h_{ik}\langle X, e_{k}\rangle, \\
\nabla_{j}\nabla_{i}H
=&-\sum_{k}h_{ijk}\langle X, e_{k}\rangle-h_{ij}-H\sum_{k}h_{ik}h_{kj}.
\endaligned
\end{equation}

\noindent
Using the above formulas and the Ricci identities, we can get the following Lemma:

\begin{lemma}\label{lemma 2.1}
Let $X:M^n\rightarrow \mathbb{R}^{n+1}$ be an $n$-dimensional complete  self-expanders in $\mathbb R^{n+1}$. We have
\begin{equation}\label{2.1-12}
\mathcal{L}H=-H(S+1),
\end{equation}
\begin{equation}\label{2.1-13}
\frac{1}{2}\mathcal{L}S
=\sum_{i,j,k}h_{ijk}^{2}-S(S+1),
\end{equation}
\begin{equation}\label{2.1-14}
\frac{1}{3}\mathcal{L}f_{3}
=2\sum_{i,j,k}h_{ijp}h_{jkp}h_{ki}-f_{3}(S+1),
\end{equation}
\end{lemma}

\noindent
\begin{lemma}\label{lemma 2.2}
Let $X:M^{n}\rightarrow \mathbb{R}^{n+1}$ be an $n$-dimensional complete self-expander in $\mathbb R^{n+1}$. If $S$ is constant,  we have
\begin{equation}\label{2.1-15}
\aligned
&\sum_{i,j,k,l}(h_{ijkl})^{2}=(S+2)\sum_{i,j,k}(h_{ijk})^{2}\\
&-6\sum_{i,j,k,l,p}h_{ijk}h_{il}h_{jp}h_{klp}
+3\sum_{i,j,k,l,p}h_{ijk}h_{ijl}h_{kp}h_{lp}.
\endaligned
\end{equation}
\end{lemma}
\begin{proof} By making use of the Ricci identities  \eqref{2.1-7},  \eqref{2.1-9} and a direct calculation, we
can  obtain \eqref{2.1-15}.
\end{proof}

In order to prove our results, we need the following generalized maximum principle
due to Omori \cite{O} and Yau \cite{Y}.

\vskip2mm
\noindent
\begin{lemma}\label{lemma 2.3}
Let $(M^{n},g)$ be a complete Riemannian manifold with Ricci curvature
bounded from below. For a $C^{2}$-function $f$ bounded from above, there exists a sequence
of points $\{p_{t}\}\in M^{n}$, such that
\begin{equation*}
\lim_{t\rightarrow\infty} f(p_{t})=\sup f,\quad
\lim_{t\rightarrow\infty} |\nabla f|(p_{t})=0,\quad
\limsup_{t\rightarrow\infty}\Delta f(p_{t})\leq 0.
\end{equation*}
\end{lemma}

 \vskip10mm
\section{Proof of Theorem 1.1}

\vskip2mm
\noindent

As we all know, if $S\equiv0$, $X:M^{3}\rightarrow \mathbb{R}^{4}$ is $\mathbb{R}^{3}$. Next, we consider that $S>0$.
We will prove the following proposition.

\begin{proposition}\label{proposition 3.1}
For a $3$-dimensional complete self-expander $X:M^{3}\rightarrow \mathbb{R}^{4}$ with  non-zero constant squared norm $S$ of the second fundamental form, if $f_{3}$ is constant,
we have that $\sup H=\frac{3f_{3}}{2S}$.
\end{proposition}

\begin{proof}
\noindent
We choose $e_{1}$ and $e_{2}$,  at each point $p\in M^{3}$,  such that
$$
h_{ij}=\lambda_i\delta_{ij}.
$$
From the definitions of $S$ and $H$, we obtain
$$
H^{2}\leq 3S.
$$
\noindent Since $S$ is constant, from the Gauss
equations, we know that the Ricci curvature of $X:M^{3} \to \mathbb{R}^{4}$
is bounded from below.
By applying the generalized maximum principle to the function $H$. Thus, there exists a sequence $\{p_{t}\}$ in $M^{3}$ such that
\begin{equation}\label{3.1-1}
\lim_{t\rightarrow\infty} H(p_{t})=\sup H,\quad
\lim_{t\rightarrow\infty} |\nabla H|(p_{t})=0,\quad
\limsup_{t\rightarrow\infty} \Delta H(p_{t})\leq 0.
\end{equation}

\noindent
Since $S$ is constant, by \eqref{2.1-13} and \eqref{2.1-15}, we know that
$\{h_{ij}(p_{t})\}$,  $\{h_{ijk}(p_{t})\}$ and $\{h_{ijkl}(p_{t})\}$ are bounded sequences for $ i, j, k, l = 1,2,3$.
Hence, we can assume that they convergence if necessary, taking a subsequence.
\begin{equation*}
\begin{aligned}
&\lim_{t\rightarrow\infty}H(p_{t})=\bar H, \ \ \lim_{t\rightarrow\infty}h_{ij}(p_{t})=\bar h_{ij}=\bar \lambda_i\delta_{ij}, \\
&\lim_{t\rightarrow\infty}h_{ijk}(p_{t})=\bar h_{ijk}, \ \  \lim_{t\rightarrow\infty}h_{ijkl}(p_{t})=\bar h_{ijkl}, \ \ i, j, k,l=1, 2, 3.
\end{aligned}
\end{equation*}

\noindent
By taking exterior derivative of $H$, from \eqref{3.1-1} and by taking limits,
we know
\begin{equation}\label{3.1-2}
\bar H_{,i}=\bar h_{11i}+\bar h_{22i}+\bar h_{33i}=0, \ \ i=1, 2, 3.
\end{equation}

\noindent
According to the definition of  the self-expander, \eqref{3.1-1} yields
$$H_{,i}=-\sum_{k}h_{ik}\langle X, e_{k}\rangle, \ \  \ \ i=1, 2, 3,$$
$$H_{,ij}=-\sum_{k}h_{ijk}\langle X, e_{k}\rangle-h_{ij}-H\sum_{k}h_{ik}h_{kj}, \ \ i,j=1, 2, 3.
$$
Thus, we get
\begin{equation}\label{3.1-3}
\bar H_{,i}=\bar h_{11i}+\bar h_{22i}+\bar h_{33i}=-\bar \lambda_{i}\lim_{t\rightarrow\infty} \langle X, e_{i} \rangle(p_{t}),\ \ \ \ i=1, 2, 3
\end{equation}
and
\begin{equation}\label{3.1-4}
\begin{cases}
\begin{aligned}
&\bar H_{,ii}=-\sum_{k}\bar h_{iik}\lim_{t\rightarrow\infty} \langle X, e_{k} \rangle(p_{t})-\bar \lambda_{i}-\bar H\bar \lambda^{2}_{i},\ \ i=1,2,3, \\
&\bar H_{,ij}=-\sum_{k}\bar h_{ijk}\lim_{t\rightarrow\infty} \langle X, e_{k} \rangle(p_{t}),\ \ i\neq j, \ \ i,j=1,2,3.
\end{aligned}
\end{cases}
\end{equation}

\noindent Since $S$ is constant, we obtain
$$
\sum_{i,j}h_{ij}h_{ijk}=0, \ \  \ k=1, 2, 3,
$$
$$
\sum_{i,j}h_{ij}h_{ijkl}+\sum_{i,j}h_{ijk}h_{ijl}=0, \ \  \ k,l=1, 2, 3.
$$
Under the processing  by taking limits, we have
\begin{equation}\label{3.1-5}
\bar\lambda_{1}\bar h_{11k}+\bar\lambda_{2}\bar h_{22k}+\bar\lambda_{3}\bar h_{33k}=0, \ \  \ k=1, 2, 3
\end{equation}
and
\begin{equation}\label{3.1-6}
\begin{cases}
\begin{aligned}
&\sum_{i}\bar \lambda_{i}\bar h_{iikk}=-\sum_{i,j}\bar h^{2}_{ijk}, \ \  \ k=1, 2, 3, \\
&\sum_{i}\bar \lambda_{i}\bar h_{iikl}=-\sum_{i,j}\bar h_{ijk}\bar h_{ijl}, \ \ k\neq l, \ \  \ k,l=1, 2, 3.
\end{aligned}
\end{cases}
\end{equation}

\noindent It follows from Ricci identities \eqref{2.1-7} that
\begin{equation*}
\bar h_{ijkl}-\bar h_{ijlk}=\bar\lambda_{i}\bar\lambda_{j}\bar\lambda_{k}\delta_{il}\delta_{jk}
-\bar\lambda_{i}\bar\lambda_{j}\bar\lambda_{l}\delta_{ik}\delta_{jl}
+\bar\lambda_{i}\bar\lambda_{j}\bar\lambda_{k}\delta_{ik}\delta_{jl}
-\bar\lambda_{i}\bar\lambda_{j}\bar\lambda_{l}\delta_{il}\delta_{jk}.
\end{equation*}
That is,
\begin{equation}\label{3.1-7}
\begin{cases}
\begin{aligned}
& \bar h_{1122}-\bar h_{2211}=\bar \lambda_{1}\bar \lambda_{2}(\bar \lambda_{1}-\bar \lambda_{2}),\ \ \bar h_{1133}-\bar h_{3311}=\bar \lambda_{1}\bar \lambda_{3}(\bar \lambda_{1}-\bar \lambda_{3}),\\
&\bar h_{2233}-\bar h_{3322}=\bar \lambda_{2}\bar \lambda_{3}(\bar \lambda_{2}-\bar \lambda_{3}), \ \ \bar h_{iikl}-\bar h_{iilk}=0, \ \ \ i,k,l=1, 2, 3.
\end{aligned}
\end{cases}
\end{equation}

\noindent Since $f_{3}$ is constant,  by \eqref{2.1-10}, we know that

\begin{equation}\label{3.1-8}
\bar \lambda^{2}_{1}\bar h_{11k}+\bar \lambda^{2}_{2}\bar h_{22k}+\bar \lambda^{2}_{3}\bar h_{33k}=0, \ \  k=1, 2, 3
\end{equation}
and
\begin{equation}\label{3.1-9}
\begin{cases}
\begin{aligned}
&\sum_{i}\bar \lambda^{2}_{i}\bar h_{iikk}=-2\sum_{i,j}\bar \lambda_{i}\bar h^{2}_{ijk}, \ \  \ k=1, 2, 3, \\
&\sum_{i}\bar \lambda^{2}_{i}\bar h_{iikl}=-2\sum_{i,j}\bar \lambda_{i}\bar h_{ijk}\bar h_{ijl}, \ \ k\neq l, \ \  \ k,l=1, 2, 3.
\end{aligned}
\end{cases}
\end{equation}

From the above equations  \eqref{3.1-2}   to \eqref{3.1-9}, we can prove the following claim that $$\sup H=\bar H=\frac{3f_{3}}{2S}.$$
In fact, if   $\bar \lambda_1=\bar \lambda_2 =\bar \lambda_3$, from \eqref{3.1-6} and \eqref{3.1-9}, we have
\begin{equation*}
\begin{cases}
\begin{aligned}
\bar \lambda_{1}\sum_{i}\bar h_{ii11}
=&-(\bar h^{2}_{111}+\bar h^{2}_{221}+\bar h^{2}_{331})-2(\bar h^{2}_{112}+\bar h^{2}_{113}+\bar h^{2}_{123}),\\
\bar \lambda_{1}\sum_{i}\bar h_{ii22}
=&-(\bar h^{2}_{112}+\bar h^{2}_{222}+\bar h^{2}_{332})-2(\bar h^{2}_{221}+\bar h^{2}_{223}+\bar h^{2}_{123}),\\
\bar \lambda_{1}\sum_{i}\bar h_{ii33}
=&-(\bar h^{2}_{113}+\bar h^{2}_{223}+\bar h^{2}_{333})-2(\bar h^{2}_{331}+\bar h^{2}_{332}+\bar h^{2}_{123})
\end{aligned}
\end{cases}
\end{equation*}
and
\begin{equation*}
\begin{cases}
\begin{aligned}
&\bar \lambda^{2}_{1}\sum_{i}\bar h_{ii11}
=-2\bar \lambda_{1}(\bar h^{2}_{111}+\bar h^{2}_{221}+\bar h^{2}_{331})-4\bar \lambda_{1}(\bar h^{2}_{112}+\bar h^{2}_{113}+\bar h^{2}_{123}),\\

&\bar \lambda^{2}_{1}\sum_{i}\bar h_{ii22}
=-2\bar \lambda_{1}(\bar h^{2}_{112}+\bar h^{2}_{222}+\bar h^{2}_{332})-4\bar \lambda_{1}(\bar h^{2}_{221}+\bar h^{2}_{223}+\bar h^{2}_{123}),\\

&\bar \lambda^{2}_{1}\sum_{i}\bar h_{ii33}
=-2\bar \lambda_{1}(\bar h^{2}_{113}+\bar h^{2}_{223}+\bar h^{2}_{333})-4\bar \lambda_{1}(\bar h^{2}_{331}+\bar h^{2}_{332}+\bar h^{2}_{123}).
\end{aligned}
\end{cases}
\end{equation*}
Thus, we infer
\begin{equation*}
\bar h_{ijk}=0, \ \ i,j,k=1, 2, 3.
\end{equation*}
Then by \eqref{2.1-13}, we know that $S=0$. It is a contradiction.

\noindent
If two of  $\bar \lambda_1$, $\bar \lambda_2$ and $\bar \lambda_3$  are equal,
without loss of generality, we assume that $\bar \lambda_{1}\neq \bar \lambda_{2}=\bar \lambda_{3}$.
Then we get from \eqref{3.1-2} and \eqref{3.1-5} that

\begin{equation}\label{3.1-10}
\bar h_{11k}=0, \ \ \bar h_{22k}+\bar h_{33k}=0,\ \ \ k=1, 2, 3.
\end{equation}
\newline
If  $\bar \lambda_{1}=0$, then $\bar \lambda_{2}=\bar \lambda_{3} \neq 0$ since $S\neq 0$.
By making use of  equations \eqref{3.1-6}, \eqref{3.1-9} and \eqref{3.1-10}, we know that

\begin{equation*}
\begin{cases}
\begin{aligned}
&\bar \lambda_{2}(\bar h_{2211}+\bar h_{3311})=-2(\bar h^{2}_{221}+\bar h^{2}_{123}),\\
&\bar \lambda_{2}(\bar h_{2222}+\bar h_{3322})=-2(\bar h^{2}_{222}+\bar h^{2}_{223})-2(\bar h^{2}_{221}+\bar h^{2}_{123})
\end{aligned}
\end{cases}
\end{equation*}
and
\begin{equation*}
\begin{cases}
\begin{aligned}
\bar \lambda^{2}_{2}(\bar h_{2211}+\bar h_{3311})=&-4\bar\lambda_{2}(\bar h^{2}_{221}+\bar h^{2}_{123}),\\

\bar \lambda^{2}_{2}(\bar h_{2222}+\bar h_{3322})=&-4\bar\lambda_{2}(\bar h^{2}_{222}+\bar h^{2}_{223})-2\bar\lambda_{2}(\bar h^{2}_{221}+\bar h^{2}_{123}).
\end{aligned}
\end{cases}
\end{equation*}
Hence, we get
\begin{equation*}
\bar h_{221}=\bar h_{123}=0, \ \ \bar h_{222}=\bar h_{223}=0.
\end{equation*}
Namely, we obtain
\begin{equation*}
\bar h_{ijk}=0, \ \ i,j,k=1, 2, 3.
\end{equation*}
Then by \eqref{2.1-13}, we know that $S=0$. It is a contradiction.

\noindent
If $\bar \lambda_{1}\neq0, \ \ \bar \lambda_{2}=\bar \lambda_{3}=0$,
combining it with \eqref{3.1-6}, \eqref{3.1-9} and \eqref{3.1-10}, we have

\begin{equation*}
\begin{cases}
\begin{aligned}
&\bar \lambda_{1}\bar h_{1111}
=-2(\bar h^{2}_{221}+\bar h^{2}_{123}),\\
&\bar \lambda_{1}\bar h_{1122}
=-2(\bar h^{2}_{222}+\bar h^{2}_{223})-2(\bar h^{2}_{221}+\bar h^{2}_{123})
\end{aligned}
\end{cases}
\end{equation*}
and
\begin{equation*}
\begin{cases}
\begin{aligned}
&\bar \lambda^{2}_{1}\bar h_{1111}=0,\\

&\bar \lambda^{2}_{1}\bar h_{1122}=-2\bar \lambda_{1}(\bar h^{2}_{221}+\bar h^{2}_{123}).
\end{aligned}
\end{cases}
\end{equation*}
We infer
\begin{equation*}
\bar h_{221}=\bar h_{123}=0, \ \ \bar h_{222}=\bar h_{223}=0.
\end{equation*}
Thus, we conclude
\begin{equation*}
\bar h_{ijk}=0, \ \ i,j,k=1, 2, 3.
\end{equation*}
Then by \eqref{2.1-13}, we know that $S=0$. It is a contradiction.

\noindent
If $\bar \lambda_{1}\neq 0, \ \ \bar \lambda_{2}=\bar \lambda_{3} \neq 0$,
combining \eqref{3.1-2} with \eqref{3.1-3}, we obtain

\begin{equation}\label{3.1-11}
\lim_{t\rightarrow\infty} \langle X, e_{k} \rangle(p_{t})=0,\ \ \ k=1, 2, 3.
\end{equation}

\noindent
It follows from \eqref{3.1-4}, \eqref{3.1-6}, \eqref{3.1-9}, \eqref{3.1-10} and \eqref{3.1-11} that

\begin{equation}\label{3.1-12}
\begin{cases}
\begin{aligned}
&\bar h_{1111}+\bar h_{2211}+\bar h_{3311}=-\bar \lambda_{1}-\bar H\bar \lambda^{2}_{1},\\
&\bar h_{1122}+\bar h_{2222}+\bar h_{3322}=-\bar \lambda_{2}-\bar H\bar \lambda^{2}_{2},\\
&\bar h_{1112}+\bar h_{2212}+\bar h_{3312}= 0,\\
&\bar h_{1113}+\bar h_{2213}+\bar h_{3313}= 0,
\end{aligned}
\end{cases}
\end{equation}

\begin{equation}\label{3.1-13}
\begin{cases}
\begin{aligned}
\bar \lambda_{1}\bar h_{1111}+\bar \lambda_{2}(\bar h_{2211}+\bar h_{3311})&=-2(\bar h^{2}_{221}+\bar h^{2}_{123}),\\
\bar \lambda_{1}\bar h_{1122}+\bar \lambda_{2}(\bar h_{2222}+\bar h_{3322})&=-2(\bar h^{2}_{221}+\bar h^{2}_{123})-2(\bar h^{2}_{222}+\bar h^{2}_{223})
,\\
\bar \lambda_{1}\bar h_{1112}+\bar \lambda_{2}(\bar h_{2212}+\bar h_{3312})&=-2(\bar h_{221}\bar h_{222}+\bar h_{223}\bar h_{123}),\\
\bar \lambda_{1}\bar h_{1113}+\bar \lambda_{2}(\bar h_{2213}+\bar h_{3313})&=-2(\bar h_{221}\bar h_{223}-\bar h_{222}\bar h_{123})
\end{aligned}
\end{cases}
\end{equation}
and
\begin{equation}\label{3.1-14}
\begin{cases}
\begin{aligned}
\bar \lambda^{2}_{1}\bar h_{1111}+\bar \lambda^{2}_{2}(\bar h_{2211}+\bar h_{3311})=&-4\bar\lambda_{2}(\bar h^{2}_{221}+\bar h^{2}_{123}),\\
\bar \lambda^{2}_{1}\bar h_{1122}+\bar \lambda^{2}_{2}(\bar h_{2222}+\bar h_{3322})=&-2(\bar\lambda_{1}+\bar\lambda_{2})(\bar h^{2}_{221}+\bar h^{2}_{123})\\
&-4\bar\lambda_{2}(\bar h^{2}_{222}+\bar h^{2}_{223}),\\
\bar \lambda^{2}_{1}\bar h_{1112}+\bar \lambda^{2}_{2}(\bar h_{2212}+\bar h_{3312})=&-4\bar\lambda_{2}(\bar h_{221}\bar h_{222}+\bar h_{223}\bar h_{123}),\\
\bar \lambda^{2}_{1}\bar h_{1113}+\bar \lambda^{2}_{2}(\bar h_{2213}+\bar h_{3313})=&-4\bar\lambda_{2}(\bar h_{221}\bar h_{223}-\bar h_{222}\bar h_{123}).
\end{aligned}
\end{cases}
\end{equation}

\noindent Besieds, \eqref{3.1-12} yields
\begin{equation}\label{3.1-15}
\bar h_{2212}+\bar h_{3312}= -\bar h_{1112},\ \
\bar h_{2213}+\bar h_{3313}= -\bar h_{1113}.
\end{equation}
Inserting \eqref{3.1-15} into \eqref{3.1-13} and \eqref{3.1-14}, we derive  to
\begin{equation}\label{3.1-16}
\bar h_{221}\bar h_{222}+\bar h_{223}\bar h_{123}=0, \ \ \bar h_{221}\bar h_{223}-\bar h_{222}\bar h_{123}=0.
\end{equation}

\noindent It follows from \eqref{3.1-13} and \eqref{3.1-14} that

\begin{equation}\label{3.1-17}
\begin{cases}
\begin{aligned}
&\bar \lambda_{1}(\bar \lambda_{2}-\bar \lambda_{1})\bar h_{1111}= 2\bar\lambda_{2}(\bar h^{2}_{221}+\bar h^{2}_{123}),\\

&\bar \lambda_{1}(\bar \lambda_{2}-\bar \lambda_{1})\bar h_{1122}= 2\bar\lambda_{2}(\bar h^{2}_{222}+\bar h^{2}_{223})+2\bar\lambda_{1}(\bar h^{2}_{221}+\bar h^{2}_{123}).
\end{aligned}
\end{cases}
\end{equation}
The following situations can be obtained from \eqref{3.1-16}:
$\bar h^{2}_{221}+\bar h^{2}_{123} =0$, or $\bar h^{2}_{222}+\bar h^{2}_{223}=0$.
\vskip1mm
If both $\bar h^{2}_{221}+\bar h^{2}_{123}=0$  and $ \bar h^{2}_{222}+\bar h^{2}_{223} =0$,
we have
\begin{equation*}
\bar h_{ijk}=0, \ \ i,j,k=1, 2, 3.
\end{equation*}
Then by \eqref{2.1-13}, we know that $S=0$. It is a contradiction.

\vskip1mm
If $\bar h^{2}_{221}+\bar h^{2}_{123}=0$ and $\bar h^{2}_{222}+\bar h^{2}_{223} \neq 0$,
by making use of \eqref{3.1-17}, we have
\begin{equation*}
\bar \lambda_{1}(\bar \lambda_{2}-\bar \lambda_{1})\bar h_{1111}= 0,\ \
\bar \lambda_{1}(\bar \lambda_{2}-\bar \lambda_{1})\bar h_{1122}= 2\bar\lambda_{2}(\bar h^{2}_{222}+\bar h^{2}_{223}).
\end{equation*}
Thus,
\begin{equation}\label{3.1-18}
\bar h_{1111}= 0, \ \ \bar h_{1122}=\frac{2\bar\lambda_{2}}{\bar \lambda_{1}(\bar \lambda_{2}-\bar \lambda_{1})}(\bar h^{2}_{222}+\bar h^{2}_{223}).
\end{equation}

\noindent
Noting that $\bar h_{1111}=0$ and by the first equation of \eqref{3.1-13}, we know
\begin{equation*}
\bar h_{2211}+\bar h_{3311}=0, \ \ \bar H_{,11}=0.
\end{equation*}
Hence, by the first equation of \eqref{3.1-12}, we have
\begin{equation}\label{3.1-19}
\bar \lambda_{1}\bar H+1=0.
\end{equation}
Combining the second equation of \eqref{3.1-12}, the second equation of \eqref{3.1-13} with \eqref{3.1-18}, we know
\begin{equation*}
\begin{aligned}
-2(\bar h^{2}_{222}+\bar h^{2}_{223})&=\bar \lambda_{1}\bar h_{1122}+\bar \lambda_{2}\Big(-\bar \lambda_{2}-\bar H\bar \lambda^{2}_{2}-\bar h_{1122}\Big) \\
&=\bar \lambda_{2}\Big(-\bar \lambda_{2}-\bar H\bar \lambda^{2}_{2}\Big)+(\bar \lambda_{1}-\bar \lambda_{2}) \cdot \frac{2\bar\lambda_{2}}{\bar \lambda_{1}(\bar \lambda_{2}-\bar \lambda_{1})}(\bar h^{2}_{222}+\bar h^{2}_{223}) \\
&=\bar \lambda_{2}\Big(-\bar \lambda_{2}-\bar H\bar \lambda^{2}_{2}\Big)-\frac{2\bar\lambda_{2}}{\bar \lambda_{1}}(\bar h^{2}_{222}+\bar h^{2}_{223}).
\end{aligned}
\end{equation*}
Then by\eqref{3.1-19}, we obtain
\begin{equation}\label{3.1-20}
\bar h^{2}_{222}+\bar h^{2}_{223}=\frac{\bar \lambda^{2}_{2}}{2}.
\end{equation}
According to \eqref{2.1-13} and \eqref{3.1-20}, we have that
\begin{equation*}
\sum_{i,j,k}h_{ijk}^{2}=S(S+1), \ \ \sum_{i,j,k}h_{ijk}^{2}=4(\bar h^{2}_{222}+\bar h^{2}_{223})=2\bar \lambda^{2}_{2}.
\end{equation*}
Thus,
$$S(S+1)=2\bar \lambda^{2}_{2}<\bar \lambda^{2}_{1}+2\bar \lambda^{2}_{2}=S.$$
It is impossible.

\vskip1mm
If  $\bar h^{2}_{221}+\bar h^{2}_{123} \neq 0$ and $ \bar h^{2}_{222}+\bar h^{2}_{223}=0$,
according to \eqref{3.1-17} , we have
\begin{equation}\label{3.1-21}
\begin{aligned}
&\bar \lambda_{1}(\bar \lambda_{2}-\bar \lambda_{1})\bar h_{1111}= 2\bar\lambda_{2}(\bar h^{2}_{221}+\bar h^{2}_{123}),\\
&\bar \lambda_{1}(\bar \lambda_{2}-\bar \lambda_{1})\bar h_{1122}= 2\bar\lambda_{1}(\bar h^{2}_{221}+\bar h^{2}_{123}).
\end{aligned}
\end{equation}
Thus,
\begin{equation}\label{3.1-22}
\bar h_{1111}= \frac{2\bar\lambda_{2}}{\bar \lambda_{1}(\bar \lambda_{2}-\bar \lambda_{1})}(\bar h^{2}_{221}+\bar h^{2}_{123}), \ \ \bar h_{1122}=\frac{2}{\bar \lambda_{2}-\bar \lambda_{1}}(\bar h^{2}_{221}+\bar h^{2}_{123}).
\end{equation}
Combining the second equation of \eqref{3.1-12}, the second equation of \eqref{3.1-13} with \eqref{3.1-22}, we know
\begin{equation*}
\begin{aligned}
-2(\bar h^{2}_{221}+\bar h^{2}_{123})&=\bar \lambda_{1}\bar h_{1122}+\bar \lambda_{2}\big(-\bar \lambda_{2}-\bar H\bar \lambda^{2}_{2}-\bar h_{1122}\big) \\
&=\bar \lambda_{2}\big(-\bar \lambda_{2}-\bar H\bar \lambda^{2}_{2}\big)+(\bar \lambda_{1}-\bar \lambda_{2})\bar h_{1122} \\
&=-\bar \lambda^{2}_{2}\big(1+\bar H\bar \lambda_{2}\big)-2(\bar h^{2}_{221}+\bar h^{2}_{123}).
\end{aligned}
\end{equation*}
Hence,
\begin{equation}\label{3.1-23}
\bar \lambda_{2}\bar H+1=0.
\end{equation}
By making use of the first equation of \eqref{3.1-12}, the first equation of \eqref{3.1-13}, \eqref{3.1-22} and \eqref{3.1-23}, we know

\begin{equation*}
\begin{aligned}
-2(\bar h^{2}_{221}+\bar h^{2}_{123})&=\bar \lambda_{1}\bar h_{1111}+\bar \lambda_{2}\big(-\bar \lambda_{1}-\bar H\bar \lambda^{2}_{1}-\bar h_{1111}\big) \\
&=\bar \lambda_{2}\big(-\bar \lambda_{1}-\bar H\bar \lambda^{2}_{1}\big)+(\bar \lambda_{1}-\bar \lambda_{2})\bar h_{1111} \\
&=-\bar \lambda_{1}\bar \lambda_{2}\big(1-\frac{\bar \lambda_{1}}{\bar \lambda_{2}}\big)-\frac{2\bar \lambda_{2}}{\bar \lambda_{1}}(\bar h^{2}_{221}+\bar h^{2}_{123}).
\end{aligned}
\end{equation*}
Thus,
\begin{equation*}
\bar h^{2}_{221}+\bar h^{2}_{123}=-\frac{\bar \lambda^{2}_{1}}{2}.
\end{equation*}
It is a contradiction.

\vskip2mm
\noindent
If  $\bar \lambda_1$, $\bar \lambda_{2}$ and $\bar \lambda_3$ are distinct,
by use of \eqref{3.1-2}, \eqref{3.1-5} and \eqref{3.1-8}, we have that
\begin{equation*}
\bar h_{11k}=\bar h_{22k}=\bar h_{33k}=0, \ \ k=1, 2, 3,
\end{equation*}
and
\begin{equation*}
\sum_{i,j,k}\bar h_{ijk}^{2}=6\bar h^{2}_{123}, \ \ 2\sum_{i,j,k,l}\bar h_{ijl}\bar h_{jkl}\bar h_{ki}
=4\bar H\bar h^{2}_{123}.
\end{equation*}
Then it follows from \eqref{2.1-13} and \eqref{2.1-14} in the Lemma \ref{lemma 2.1} that
\begin{equation*}
6\bar h^{2}_{123}=S(S+1), \ \ 4\bar H\bar h^{2}_{123}=(S+1)f_{3},
\end{equation*}
namely,
\begin{equation*}
 (2\bar H S-3f_{3})(S+1)=0.
\end{equation*}
Hence,
\begin{equation*}
\bar H =\frac{3f_{3}}{2S}.
\end{equation*}

\end{proof}

By applying the generalized maximum principle due to Omori and Yau to the function $-H$, we can obtain the following
\begin{proposition}\label{proposition 3.2}
For a $3$-dimensional complete self-expander $X:M^{3}\rightarrow \mathbb{R}^{4}$ with  non-zero constant squared norm $S$ of the second fundamental form, if $f_{3}$ is constant,
we have that $\inf H=\frac{3f_{3}}{2S}$.
\end{proposition}

\vskip3mm
\noindent
{\it Proof of Theorem \ref{theorem 1.1}}.
If $S=0$, we know that $X: M^{3}\to \mathbb{R}^{4}$ is a hyperplane
$\mathbb {R}^{3}$ through the origin. If $S\neq0$,
from the Proposition \ref{proposition 3.1} and the Proposition \ref{proposition 3.2}, we have that $\sup H=\inf H=\frac{3f_{3}}{2S}$. That is, the mean curvature $H$ and
the principal curvatures are constant. Then by a classification theorem due to Lawson \cite{Law}, $X: M^{3}\to \mathbb{R}^{4}$ is $\mathbb{S}^{k}(r)\times \mathbb{R}^{3-k}, \ \ k=1,2,3$. However, for the constant mean curvature $H$, we infer that $H=0$ from \eqref{2.1-12}. It is a contradiction. So the main theorem of the present paper is proved.

\begin{flushright}
$\square$
\end{flushright}

\noindent{\bf Acknowledgements}
The first author was partially supported by the China Postdoctoral Science Foundation Grant No.2022M711074. The second author was partly supported by grant No.12171164 of NSFC, GDUPS (2018), Guangdong Natural Science Foundation Grant No.2023A1515010510.

\noindent{\bf Declarations}

\noindent{\bf Conflict of interest} There are no conflicts of interest with third parties.

\noindent{\bf Data availability} Data sharing not applicable to this article as no datasets were generated or analysed during the current study.

\end{document}